\documentclass[11pt]{amsart}

\usepackage{amsmath,amssymb,amsfonts,euscript,enumerate,amsthm}
\usepackage{color}

\usepackage{tikz}

\usepackage{graphicx}
\usepackage{epstopdf}
\usepackage[colorlinks=true,citecolor=red,linkcolor=blue]{hyperref}
\usepackage[]{cleveref}
\usepackage{mathtools}
\usepackage{esint}
\usepackage{mathrsfs}
\usepackage[notref, notcite, final]{showkeys}

\theoremstyle{plain}
\newtheorem{theorem}{Theorem}[section]
\newtheorem{lemma}[theorem]{Lemma}

\newtheorem{thmx}{Theorem}

\theoremstyle{definition}
\newtheorem{definition}[theorem]{Definition}

\newtheorem*{conjecture}{Conjecture}

\theoremstyle{remark}

\numberwithin{equation}{section}

\newcommand{\R}{{\mathbb R}}
\newcommand{\bS}{\mathbb{S}}

\newcommand{\ga}{\gamma}

\newcommand{\e}{\varepsilon}

\newcommand{\Om}{\Omega}
\newcommand{\La}{\Lambda}

\newcommand{\Ga}{\Gamma}

\newcommand{\ti}{\times}

\newcommand{\pa}{\partial}
\newcommand{\su}{\subset}
\newcommand{\qu}{\quad}

\newcommand{\sm}{\setminus}
\newcommand{\ra}{\rightarrow}

\newcommand{\D}{\nabla}

\newcommand{\fr}{\frac}

\newcommand{\inn}[2]{\left\langle {#1},{#2} \right\rangle}
\newcommand{\norm}[1]{\left\lVert#1\right\rVert}

\newcommand{\bu}{\mathbf{u}}
\newcommand{\bv}{\mathbf{v}}
\newcommand{\cO}{\mathcal{O}}
\newcommand{\cA}{\mathcal{A}}

\makeatletter
\@namedef{subjclassname@2020}{\textup{2020} Mathematics Subject Classification}
\makeatother

\title[Energy-minimizing curves in constrained spaces]{On the uniqueness of energy-minimizing curves in constrained spaces}

\author{Ki-Ahm Lee}
\address{Department of Mathematical Sciences and Research Institute of Mathematics, Seoul National University, Seoul 08826, Republic of Korea}
\email{kiahm@snu.ac.kr}

\author{Taehun Lee}
\address{School of Mathematics, Korea Institute for Advanced Study, Seoul 02455, Republic of Korea}
\email{taehun@kias.re.kr}

\subjclass[2020]{53C43 (Primary) 58E20, 53A04 (Secondary)}
\keywords{uniqueness, harmonic map, variational inequality, geodesic}

\begin{document}

\begin{abstract}	
In this paper, we investigate energy-minimizing curves with fixed endpoints $p$ and $q$ in a constrained space. 
We prove that when one of the endpoints, say $p$, is fixed, the set of points $q$ for which the energy-minimizing curve is not unique has no interior points.
\end{abstract}

\maketitle

\section{Introduction}
Let $\cO$ be a bounded convex domain in $\R^n$ with smooth boundary $\pa \cO$, and let $p,q$ be two points in $\R^n\sm \overline\cO$. In this paper, we are concerned with energy minimizing curves joining two points $p$ and $q$ in $\cO^c:=\R^n\sm \cO$. More precisely, we investigate minimizers of the energy 
\begin{align}\label{eq:E}
E(\bu)=\int_I |\bu'(t)|^2 \text{d}t
\end{align}
in the admissible set 
\begin{align}\label{eq:cA}
\cA = \left\{ \bu \in H^1(I,\R^n) \vert \,\,\bu(0)=p, \, \bu(1)=q, \,\bu(t) \in \cO^c \text{ for all }t\in I\right\},
\end{align}
where $I=[0,1]$ is the unit interval and $H^1$ is the Sobolev space of vector-valued functions. 
The set $\cO$ is the so-called \textit{obstacle} since functions in $\cA$ are forbidden from entering $\cO$.

To avoid trivial cases, we assume throughout that the line segment $l_{pq}$ joining $p$ to $q$ intersects the obstacle $\cO$, i.e., $l_{pq}\cap \cO\not=\emptyset$.

The first result of this paper shows existence of a minimizer and the optimal regularity of minimizers. A simple example where the minimizer is not $C^2$ can be observed when $\pa\cO$ is $\bS^{n-1}$. This scenario will be discussed in the final section, \Cref{sec:example}.

\begin{theorem}\label{thm:exist-reg}
Let $\mathcal{O}$ be a bounded convex domain in $\R^{n}$ with smooth boundary $\pa\cO$, and let $p,q$ be two points in $\R^n\sm\overline\cO$. Then there exists a minimizer of \eqref{eq:E} in the admissible set $\cA$, and any minimizer is a $C^1$ curve with bounded curvature depending only on the maximum principal curvature of $\pa\cO$.
\end{theorem}

An interesting question is whether minimizers are unique or not. 
The simple example shows that both cases occur. Indeed, when the obstacle is given as the standard unit ball in $\R^{n}$ so that $\pa\cO=\mathbb{S}^{n-1}$, if the line segment $l_{pq}$ passing through the given two points also passes through the origin, then all rotations of a minimizer around this segment become minimizers, resulting in infinitely many minimizers. If this condition is not met, then the minimizer is unique. Detailed discussion on this example can be found in \Cref{sec:example}.

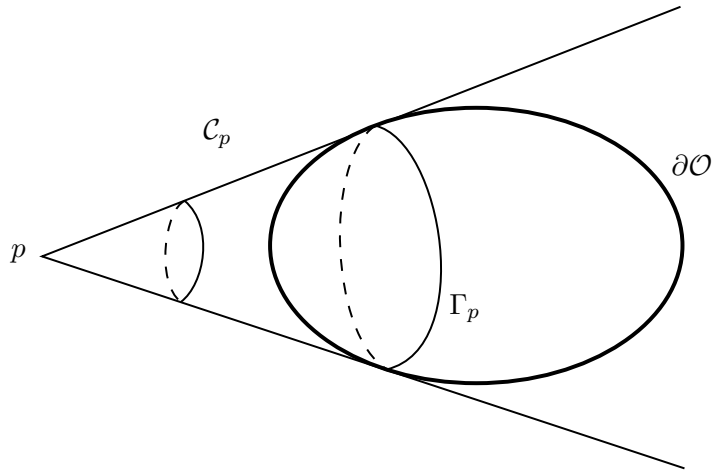
\begin{figure}\label{fig:visions}

\tikzset{every picture/.style={line width=0.75pt}} 

\begin{tikzpicture}[x=0.75pt,y=0.75pt,yscale=-1,xscale=1]

\draw  [line width=1.5]  (249,135.5) .. controls (249,97.12) and (295.56,66) .. (353,66) .. controls (410.44,66) and (457,97.12) .. (457,135.5) .. controls (457,173.88) and (410.44,205) .. (353,205) .. controls (295.56,205) and (249,173.88) .. (249,135.5) -- cycle ;
\draw    (456,15) -- (134,141) -- (458,248) ;
\draw    (302,75) .. controls (336,83) and (353,187) .. (308,198) ;
\draw  [dash pattern={on 4.5pt off 4.5pt}]  (302,75) .. controls (276,93) and (279,181) .. (308,198) ;
\draw    (206,113) .. controls (221,126) and (216,155) .. (204,164) ;
\draw  [dash pattern={on 4.5pt off 4.5pt}]  (206,113) .. controls (196,117) and (191,154) .. (204,164) ;

\draw (117,133.4) node [anchor=north west][inner sep=0.75pt]    {$p$};
\draw (214,70.4) node [anchor=north west][inner sep=0.75pt]    {$\mathcal{C}_{p}$};
\draw (338,158.4) node [anchor=north west][inner sep=0.75pt]    {$\Gamma _{p}$};
\draw (450,88.4) node [anchor=north west][inner sep=0.75pt]    {$\partial \mathcal{O}$};

\end{tikzpicture}

\caption{Vision cone $\mathcal{C}_p$ and vision boundary $\Ga_p$.}

\end{figure}

The example of the standard sphere suggests that non-uniqueness can arise only under specific circumstances. The following result partially confirms this conjecture.

\begin{theorem}\label{thm:unique}
Fix a point $p$ in $\R^n\sm \overline\cO$. 
Let $\mathcal{S}_1$ and $\mathcal{S}_{2}$ denote the sets of points $q$ in $\R^n\sm \overline \cO$ for which the minimizer of the energy \eqref{eq:E} over $\cA$ is unique and non-unique, respectively. Then $\mathcal{S}_2 \su \pa \mathcal{S}_1$, and in particular $\operatorname{int}(\mathcal{S}_2)=\emptyset$.
\end{theorem}

The derivation of \Cref{thm:unique} relies on a particular characterization of minimizers, necessitating the introduction of the following definition.

\begin{definition}\label{def:vision}
Let $\cO$ be a bounded convex domain in $\R^n$. 
A cone with apex $p$ is called the \textit{vision cone} of $p$ (relative to $\cO$) if it is tangent to $\cO$. The intersection of the vision cone and $\pa\cO$ is called the \textit{vision boundary} of $p$. The vision cone and vision boundary will be denoted by $\mathcal{C}_p$ and $\Ga_p$.
\end{definition}

The characterization of the minimizer can be given as follows:
\begin{theorem}\label{thm:charac}
Any minimizer of \eqref{eq:E} over $\cA$ consists of a geodesic on $\pa\cO$ that connects two vision boundary $\Ga_p$ and $\Ga_q$, with two line segments connecting the points $p$ and $q$ to the vision boundaries $\Ga_p$ and $\Ga_q$, respectively.
\end{theorem}

\bigskip

The energy minimizers of \eqref{eq:E}, unrestricted to $\cA$, correspond to one-dimensional harmonic maps. Comprehensive overviews on harmonic maps can be found in \cite{SU82_JDG,SU83_JDG,SU84_Invent}. 
Additional free boundary problems that minimize the energy function in \eqref{eq:E} with supplemental terms have been discussed in \cite{ASUW15_AM,CSY18_Duke}.
The notion of energy minimizing constrained maps has been investigated for decades, see \cite{Duzaar87_JRAM,Fuchs89_CPDE}. The optimal regularity of projected maps recently addressed in \cite{FKS23_arxiv}.

An analogous problem for surfaces poses an intriguing question. That is, whether for almost every given curve $\ga$, there exists a unique energy-minimizing surface that takes $\ga$ as its boundary. Here, the term "almost every" used for curves is as described in \cite{Morgan78_Invent}.

The paper is organized as follows: In \Cref{sec:pre}, we recall some preliminaries and summarize known results related to the regularity of minimizers in constrained spaces. The existence of minimizers and the proof of \Cref{thm:exist-reg} are covered in \Cref{sec:existence}. \Cref{sec:optimal} is devoted to examining necessary properties for the characterization of minimizers and includes the proof of \Cref{thm:charac}. Our main result, \Cref{thm:unique}, is proved in \Cref{sec:uniqueness}. Lastly, in \Cref{sec:example}, we discuss the special case where the constraint is the standard sphere, and propose a conjecture regarding the Hausdorff dimension of the non-uniqueness set $\mathcal{S}_2$.

\section{Preliminaries}\label{sec:pre}
Let $A$ denotes the second fundamental form of $\pa\cO$.
The Euler--Lagrange equation for \eqref{eq:E} over $\cA$ is
\begin{align}\label{eq:EL}
\bu'' = A(\bu',\bu')\chi_{\{\bu\in \pa\cO\}}\qu \text{in } I
\end{align}
in the sense of distributions (see \cite{Duzaar87_JRAM}).

Since $\cO$ is convex domain, the projection from $\cO^c$ onto $\pa \cO$ are well-defined. Moreover, since $\pa\cO$ is smooth, the outward unit normal $\nu$ is also well-defined. Given a curve $\bu$ in $\cA$, it can be expressed as
\begin{align}
\bu = \bv + hN,
\end{align}
where $\bv:I\ra\pa\cO$ is the projection of $\bu$ onto $\pa\cO$, $h:I\ra \R$ denotes the distance function $|\bu-\bv|$, and $N$ is the outward unit normal to $\pa\cO$.

Let $\bu:I\ra \pa\cO$ be a curve. Then the length of the curve is given by
\begin{align}
L(\bu)= \int_I |\bu'(t)| dt
\end{align}
By using the Cauchy--Schwarz inequality,
\begin{align}\label{ineq:LE}
L(\bu)^2 \le E(\bu)
\end{align}
and equality occurs if and only if $|\bu'(t)|$ is constant. Using these, if a curve $\mathbf{w}$ minimizes energy, then $\textbf{w}$ minimizes its length.

We close this section by recalling regularity results for minimizers in constrained spaces.
\begin{thmx}[\cite{FKS23_arxiv}]\label{thm:reg}
Let $u\in \cA$ be a energy minimizer of \eqref{eq:E} over $\cA$. Then $u\in C^{1,1}_{loc}(I)$.
\end{thmx}

\section{Existence of minimizer}\label{sec:existence}

In this section, we establish the existence of a minimizer for \eqref{eq:E} over the admissible set $\cA$ defined in \eqref{eq:cA} and present the proof for \Cref{thm:exist-reg}.

\begin{lemma}\label{lem:exist}
There exists a solution $\bu:[0,1]\ra \cO^c$ minimizing energy \eqref{eq:E} over the admissible set $\mathcal{A}$.
\end{lemma}

\begin{proof}
Since $\cO$ is bounded, any two points outside $\cO$ are connected by a smooth curve in $\cO^c$, which induces that $\cA$ is not empty. We now take a minimizing sequence $\bu_k \in \cA$ such that 
\begin{align}
\lim_{k\ra \infty} E(\bu_k) = \inf_{\bv\in\cA} E(\bv).
\end{align}

Let $c:I\ra \cO^c$ be a smooth curve connecting two points $p$ and $q$. Since $c\in \cA$, we have $\inf_{\bv\in\cA}E(\bv)\le E(c)<\infty$, and thus $E(\bu_k)$ is uniformly bounded. Using the Poincar\`e inequality, $\norm{\bu_k}_{H^1}$ is also uniformly bounded. Then there exists a function $\bu\in H^1(I;\R^n)$ such that a subsequence of $\bu_k$ converges weakly to $\bu$ in $H^1(I;\R^n)$ and converges (strongly) to $\bu$ almost everywhere in $I$. We still denote by $\bu_k$ the converging subsequence. 

Clearly, $\bu\in \cA$ since $\bu_k(0)=p$, $\bu_k(1)=q$, and $\bu_k(I)$ is contained in the closed set $\cO^c$. Moreover, since $E$ is weakly lower semicontinuous on $H^1(I;\R^n)$, we have
\begin{align}
E(\bu)\le\liminf_{k\ra\infty} E(\bu_k)=\inf_{\bv\in\cA}E(\bv)\le E(\bu)
\end{align}
which conclude that $\bu$ is a minimizer of $E$ on $\cA$.
\end{proof}

\begin{proof}[Proof of \Cref{thm:exist-reg}]
The result follows from \Cref{thm:reg}, \Cref{lem:exist}, and the fact that $\bu''=0$ holds away from the obstacle, as stated in \eqref{eq:EL}.
\end{proof}

\section{Proof of \Cref{thm:charac}}\label{sec:optimal}
Let $\La$ and $\Om$ be the coincidence set and the non-coincidence set, i.e., $\La=\{t\in I:\bu(t)\in \pa\mathcal{O}\}$ and $\Om=I\sm \La=\{t\in I:\bu(t)\not\in \pa\cO\}$. For convenience, we assume that the origin lies in $\cO$. We also assume that the line segment $l_{pq}$ passing through $p$ and $q$ intersects with $\cO$. Otherwise, the minimizer is exactly the line segment $l_{pq}$ passing through $p$ and $q$.

\begin{lemma}\label{lem:La}
The coincidence set $\La$ is of the form $[a,b]$ for some $a,b \in (0,1)$.
\end{lemma}

\begin{proof}
Since the non-coincidence set $\Om=\{t\in[0,1]: \bu(t)\in \R^n\sm \overline{\mathcal{O}}\}$ is open in $[0,1]$ and $\{0,1\}\su \Om$, the non-coincidence set $\Om$ is of the form 
\begin{align}
\Om=[0,a)\cup \bigcup_{k=1}^m(a_k,b_k)\cup(b,1]
\end{align}
for some $m\in \mathbb{N}\cup \{0,\infty\}$. We claim that $m=0$. If $m>0$, then for each $(a_k,b_k)$, it holds
\begin{align}
\bu(t)\notin \cO \text{ for } t\in (a_k,b_k) \qu \text{and} \qu \bu(a_k),\bu(b_k)\in \pa\cO.
\end{align}

Let $\bv:[a_k,b_k]\ra \pa \cO$ be the projected curve of $\bu\vert_{[a_k,b_k]}$ into $\pa\cO$. Then, we can express $\bu$ as $\bu=\bv+hN$, where $N$ is the outward unit normal to $\pa\cO$ and $h:[0,1] \ra \R$ is a nonnegative function satisfying $h(a_k)=h(b_k)=0$ and $h(t)>0$ for $a_k<t<b_k$ so that
\begin{align}
\int_{a_k}^{b_k} |h'|^2>0.
\end{align}
Hence, noting $\bv'\cdot N=0$ and $N'\cdot N=0$, we have
\begin{equation}\label{eq:Eu-Ev}
\begin{split}
\int_{a_k}^{b_k}|\bu'(t)|^2-|\bv'(t)|^2&=2\int_{a_k}^{b_k} \bv'\cdot (hN)' +\int_{a_k}^{b_k} |(hN)'|^2
\\
&=
2\int_{a_k}^{b_k} (\bv'\cdot N')h+\int_{a_k}^{b_k} |h'|^2 + h^2|N'|^2 >0
\end{split}
\end{equation}
since $\bv'\cdot N'=-\bv''\cdot N \ge 0$ by the convexity of the obstacle $\cO$ and $h> 0$ in $(a_k,b_k)$. 

Let $\tilde{\bu}$ be the function derived from $\bu$ by replacing $\bu$ with $\bv$ over the interval $(a_k,b_k)$. Then it follows from \eqref{eq:Eu-Ev} that $E(\tilde{\bu}) < E(\bu)$. This contradicts the assumption that $\bu$ minimizes energy. Therefore, we must have $m=0$, which implies that $\Omega = [0, a) \cup (b,1]$ and $\La=[a,b]$, as desired.
\end{proof}

Note that we will show $a<b$ by combining the fact $\bu([0,a])$ and $\bu([b,1])$ lie on the vision cones $\mathcal{C}_p$ and $\mathcal{C}_q$, respectively.

\begin{lemma}\label{lem:u-Ga}
Let $\Ga_p$ be the vision boundary of $p$ with respect to $\cO$ (\Cref{def:vision}). Then $\bu(a) \in \Ga_p$ and $\bu([0,a])$ is the line segment connecting $p$ and $\bu(a)$. Similarly, $\bu(b) \in \Ga_q$ and $\bu([b,1])$ is the line segment connecting $\bu(b)$ and $q$.
\end{lemma}

\begin{proof}
On the non-coincidence set $[0,a)$, it follows from the Euler--Lagrange equation that $\bu''=0$. Thus $\bu([0,a])$ is the line segment connecting $p$ and $\bu(a)$.

To prove $\bu(a)\in \Ga_p$, we note that $\Ga_p$ divides $\pa\cO$ into two parts. Let us denote the piece that is closer to the point $p$ as $\pa\cO_p$.
If $\bu(a)\not\in \Ga_p$, then $\bu(a)\in \pa\cO_p$.
Since $l_{pq}\cap \cO\not=\emptyset$, we see that $\bu(b)$ lies in the other piece of $\pa\cO$. Thus there exists $t_0\in (a,b)$ such that $\bu(t_0)\in \Ga_p$. However, replacing $\bu([0,t_0])$ with the line segment connecting $p$ and $\bu(t_0)$ reduces energy, which is a contradiction. Therefore, $\bu(a)\in \Ga_p$. 

The same argument shows the similar statement for $q$.
\end{proof}

\begin{lemma}
The coincidence set $\La$ is not a single point, i.e., $a<b$.
\end{lemma}

\begin{proof}
Suppose $a=b$. Given the smoothness of $\partial \cO$, $\bu(a)$ has a unique tangent plane of $\pa\cO$ at $\bu(a)$. By virtue of \Cref{lem:u-Ga}, both points $p$ and $q$ are situated on this tangent plane. This, however, leads to a contradiction, as it implies that the line segment $l_{pq}$ has no intersection with $\cO$.
\end{proof}

To prove \Cref{thm:charac}, it remains to show that the coincidence parts of $\bu$ is a geodesic on the obstacle.

\begin{lemma}\label{lem:geo}
The restriction $\bu\vert_{[a,b]}$ is a minimizing geodesic joining $\bu(a)$ and $\bu(b)$ on $\pa\cO$.
\end{lemma}

\begin{proof}
Let $f:[a,b]\ra\R$ be a function with $f(t)>0$ in $(a,b)$ and $f(a)=f(b)=0$, and let $V(t)=f(t) \D_{\bu'}\bu'$.
Consider a variation $F:(-\e,\e)\ti [a,b]\ra \pa\cO$ of $\bu\vert_{[a,b]}$ having $V(t)$ as variational field, i.e., $F(0,t)=\bu\vert_{[a,b]}$, $F(\cdot,a)\equiv\bu(a)$, $F(\cdot,b)\equiv\bu(b)$, and $\fr{\pa F}{\pa s}(0,\cdot)=V$.

Let $E(s)$ be the energy of the perturbed curve $F(s,\cdot)$. By differentiating 
\begin{align}
E(s) = \int_a^b |\pa_tF(s,t)|^2 dt
\end{align}
with respect to $s$, we have
\begin{align}
\pa_s \int_a^b |\pa_tF(s,t)|^2 dt = 2\int_a^b \inn{\D_{s}\pa_t F}{\pa_t F}dt=2\int_a^b \inn{\D_{t}\pa_s F}{\pa_t F}dt.
\end{align}
Using the integration by parts, we obtain
\begin{align}
\int_a^b \inn{\D_{t}\pa_s F}{\pa_t F}dt
&=
\int_a^b \pa_t\inn{\pa_s F}{\pa_t F}dt - \int_a^b \inn{\pa_s F}{\D_{t}\pa_t F}dt
\\
&=\inn{\pa_s F}{\pa_t F}\vert_a^b - \int_a^b \inn{\pa_s F}{\D_{t}\pa_t F}dt.
\end{align}
Taking $s=0$, we conclude that
\begin{align}
0=\tfrac{1}{2}E'(0)&= \inn{V(b)}{\bu(b)}-\inn{V(a)}{\bu(a)}- \int_a^b \inn{V(t)}{\D_{\bu'}\bu'}dt
\\
&=-\int_a^b f(t)|\D_{\bu'}\bu'|^2dt.
\end{align}
Therefore, we have $\D_{\bu'}\bu'=0$ in $(a,b)$, meaning that $\bu$ is indeed a geodesic.

The fact that this geodesic is minimizing follows from the equality condition given in \eqref{ineq:LE}. This completes the proof.
\end{proof}

\begin{proof}[Proof of \Cref{thm:charac}]
The proof follows from \Cref{lem:La}-\Cref{lem:geo}.
\end{proof}

\section{Proof of \Cref{thm:unique}}\label{sec:uniqueness}
The minimizer of \eqref{eq:E} is not unique in general. For example, when the obstacle $\cO$ is the unit ball in $\R^n$ such that $\pa\cO=\bS^{n-1}$, the number of solutions can be characterized as follows:
\begin{enumerate}
\item If the line passing through two points $p$ and $q$ passes the origin, then there are infinitely many minimizers.
\item Otherwise, there exists only one minimizer.
\end{enumerate}
As highlighted in the introduction, this example hints at the potential rarity of configurations for $p$ and $q$ that yield non-unique minimizers.

In this section, we prove a stronger statement than \Cref{thm:unique}. 

\begin{lemma}\label{lem:unique}
Fix a point $p$ in $\cO^c$. Assume that $\bu\in \cA$ is a minimizer of \eqref{eq:E}. Then the energy-minimizing curve in $\cA$ with respect to $p$ and $q'$, where $q'$ is an intermediate point between $\bu(b)$ and $q$, is unique.
\end{lemma}

\begin{proof}
Consider the vision boundaries $\Ga_q$ and $\Ga_{q'}$. Note that $\Ga_q$ partitions $\pa\cO$ into two distinct regions, with $\Ga_{q'}\setminus \bu(b)$ contained in one of these regions. Furthermore, observe that $\bu(b)$ lies in the intersection of $\Ga_q$ and $\Ga_{q'}$.

We begin by noting that
\begin{align}\label{eq:diff-min}
|q'-\bu(b)|-|q-\bu(b)|=-|q'-q|.
\end{align}
Select a point $x'$ in $\Ga_{q'}$, distinct from $\bu(b)$, and let $x$ denote the intersection of $\Ga_q$ and the geodesic starting at $x'$, directed towards $x'-q'$. Define $l$ as the length of the minimizing geodesic that connects $x'$ and $x$. By the triangle inequality in the triangle $xqq'$, we find
\begin{align}\label{eq:diff-any}
(|q'-x'|+l)-|q-x|>|q'-x|-|q-x|>-|q'-q|.
\end{align}

By comparing \eqref{eq:diff-min} with \eqref{eq:diff-any}, and using \Cref{thm:charac}, we deduce the uniqueness of the energy minimizer for the given fixed points $p$ and $q'$, thereby concluding the proof.
\end{proof}

\begin{proof}[Proof of \Cref{thm:unique}]
For any point $q$ in $\mathcal{S}_2$, \Cref{lem:unique} implies the existence of a line segment contained in $\mathcal{S}_1$. Therefore, $q\in \pa \mathcal{S}_1$. The assertion $\operatorname{int}\mathcal{S}_2=\emptyset$ naturally follows from $\mathcal{S}_2\subseteq \pa \mathcal{S}_1$.
\end{proof}

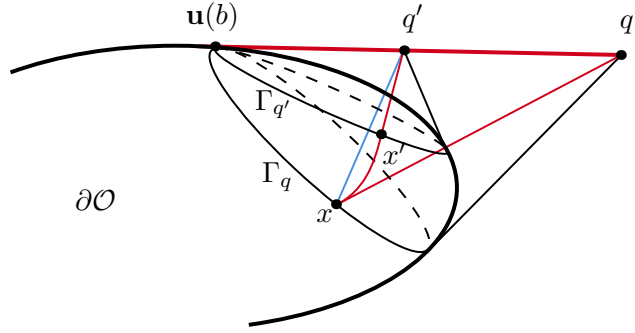
\begin{figure}\label{fig:q'}

\tikzset{every picture/.style={line width=0.75pt}} 

\begin{tikzpicture}[x=0.75pt,y=0.75pt,yscale=-1,xscale=1]

\draw [line width=0.75]    (390.67,66) -- (293.36,164.68) ;
\draw [color={rgb, 255:red, 208; green, 2; blue, 27 }  ,draw opacity=1 ][line width=0.75]    (390.67,66) -- (246.83,141.67) ;
\draw [color={rgb, 255:red, 74; green, 144; blue, 226 }  ,draw opacity=1 ][line width=0.75]    (280.83,63.33) -- (246.83,141.67) ;
\draw [color={rgb, 255:red, 208; green, 2; blue, 27 }  ,draw opacity=1 ][line width=0.75]    (280.83,63.33) -- (269.5,106.33) ;
\draw [color={rgb, 255:red, 208; green, 2; blue, 27 }  ,draw opacity=1 ][line width=1.5]    (390.67,66) -- (177.18,61.57) ;
\draw  [draw opacity=0] (292.35,164.68) .. controls (292.35,164.68) and (292.35,164.68) .. (292.35,164.68) .. controls (285.16,169.64) and (255.18,150.94) .. (225.37,122.92) .. controls (195.57,94.89) and (177.23,68.14) .. (184.42,63.18) -- (238.38,113.93) -- cycle ; \draw   (292.35,164.68) .. controls (292.35,164.68) and (292.35,164.68) .. (292.35,164.68) .. controls (285.16,169.64) and (255.18,150.94) .. (225.37,122.92) .. controls (195.57,94.89) and (177.23,68.14) .. (184.42,63.18) ;  
\draw  [draw opacity=0][dash pattern={on 4.5pt off 4.5pt}] (184.42,63.18) .. controls (184.42,63.18) and (184.42,63.18) .. (184.42,63.18) .. controls (191.6,58.21) and (221.59,76.91) .. (251.39,104.94) .. controls (281.2,132.97) and (299.53,159.71) .. (292.35,164.68) -- (238.38,113.93) -- cycle ; \draw  [dash pattern={on 4.5pt off 4.5pt}] (184.42,63.18) .. controls (184.42,63.18) and (184.42,63.18) .. (184.42,63.18) .. controls (191.6,58.21) and (221.59,76.91) .. (251.39,104.94) .. controls (281.2,132.97) and (299.53,159.71) .. (292.35,164.68) ;  

\draw  [line width=3.75] [line join = round][line cap = round] (281,63.67) .. controls (281,63.67) and (281,63.67) .. (281,63.67) ;
\draw  [draw opacity=0] (302.95,115.57) .. controls (299.62,117.87) and (270.76,107.99) .. (238.5,93.5) .. controls (206.25,79.01) and (182.8,65.4) .. (186.14,63.09) -- (244.54,89.33) -- cycle ; \draw   (302.95,115.57) .. controls (299.62,117.87) and (270.76,107.99) .. (238.5,93.5) .. controls (206.25,79.01) and (182.8,65.4) .. (186.14,63.09) ;  
\draw  [draw opacity=0][dash pattern={on 4.5pt off 4.5pt}] (186.14,63.09) .. controls (186.14,63.09) and (186.14,63.09) .. (186.14,63.09) .. controls (189.47,60.79) and (218.33,70.67) .. (250.58,85.16) .. controls (282.84,99.65) and (306.29,113.26) .. (302.95,115.57) -- (244.54,89.33) -- cycle ; \draw  [dash pattern={on 4.5pt off 4.5pt}] (186.14,63.09) .. controls (186.14,63.09) and (186.14,63.09) .. (186.14,63.09) .. controls (189.47,60.79) and (218.33,70.67) .. (250.58,85.16) .. controls (282.84,99.65) and (306.29,113.26) .. (302.95,115.57) ;  

\draw [line width=0.75]    (302.95,115.57) -- (280.83,63.33) ;
\draw [color={rgb, 255:red, 208; green, 2; blue, 27 }  ,draw opacity=1 ][line width=0.75]    (269.5,106.33) .. controls (265.83,122) and (260.83,134.33) .. (246.83,141.67) ;
\draw  [line width=3.75] [line join = round][line cap = round] (185.67,61.67) .. controls (185.67,61.67) and (185.67,61.67) .. (185.67,61.67) ;
\draw  [line width=3.75] [line join = round][line cap = round] (269.33,106) .. controls (269.33,106) and (269.33,106) .. (269.33,106) ;
\draw  [line width=3.75] [line join = round][line cap = round] (246.67,141.33) .. controls (246.67,141.33) and (246.67,141.33) .. (246.67,141.33) ;
\draw  [line width=3.75] [line join = round][line cap = round] (390.33,66) .. controls (390.33,66) and (390.33,66) .. (390.33,66) ;
\draw  [draw opacity=0][line width=1.5]  (82.51,74.8) .. controls (105.86,66.33) and (134.53,61.33) .. (165.5,61.33) .. controls (244.02,61.33) and (307.67,93.42) .. (307.67,133) .. controls (307.67,166.09) and (263.17,193.95) .. (202.69,202.19) -- (165.5,133) -- cycle ; \draw  [line width=1.5]  (82.51,74.8) .. controls (105.86,66.33) and (134.53,61.33) .. (165.5,61.33) .. controls (244.02,61.33) and (307.67,93.42) .. (307.67,133) .. controls (307.67,166.09) and (263.17,193.95) .. (202.69,202.19) ;  

\draw (388.3,44.6) node [anchor=north west][inner sep=0.75pt]    {$q$};
\draw (278.5,39.07) node [anchor=north west][inner sep=0.75pt]    {$q'$};
\draw (169.5,38.07) node [anchor=north west][inner sep=0.75pt]    {$\mathbf{u}( b)$};
\draw (113.67,132.73) node [anchor=north west][inner sep=0.75pt]    {$\partial \mathcal{O}$};
\draw (208,118.73) node [anchor=north west][inner sep=0.75pt]    {$\Gamma _{q}$};
\draw (204.33,82.07) node [anchor=north west][inner sep=0.75pt]    {$\Gamma _{q'}$};
\draw (268.67,108.4) node [anchor=north west][inner sep=0.75pt]    {$x'$};
\draw (235.33,142.73) node [anchor=north west][inner sep=0.75pt]    {$x$};

\end{tikzpicture}

\caption{Comparison of lengths: $xq'$ versus $xq$.}
\end{figure}

\section{Example: Spherical obstacle}\label{sec:example}
In this last section, we deal with the case where the obstacle $\pa\cO$ has a specific shape: the unit sphere $\bS^{n-1}$. 

\begin{figure}\label{fig:non-unique}

\tikzset{every picture/.style={line width=0.75pt}} 

\begin{tikzpicture}[x=0.75pt,y=0.75pt,yscale=-1,xscale=1]

\draw  [line width=1.5]  (163.5,102) .. controls (163.5,60.85) and (196.85,27.5) .. (238,27.5) .. controls (279.15,27.5) and (312.5,60.85) .. (312.5,102) .. controls (312.5,143.15) and (279.15,176.5) .. (238,176.5) .. controls (196.85,176.5) and (163.5,143.15) .. (163.5,102) -- cycle ;
\draw  [line width=0.75]  (22,102) -- (209.86,32.62) -- (209.86,32.62) -- cycle ;
\draw [line width=0.75]    (22,102) -- (209.41,170.89) ;
\draw  [draw opacity=0][dash pattern={on 4.5pt off 4.5pt}] (208.36,168.95) .. controls (208.36,168.95) and (208.36,168.95) .. (208.36,168.95) .. controls (208.36,168.95) and (208.36,168.95) .. (208.36,168.95) .. controls (194.27,168.95) and (182.86,138.95) .. (182.86,101.95) .. controls (182.86,64.94) and (194.27,34.95) .. (208.36,34.95) -- (208.36,101.95) -- cycle ; \draw  [dash pattern={on 4.5pt off 4.5pt}] (208.36,168.95) .. controls (208.36,168.95) and (208.36,168.95) .. (208.36,168.95) .. controls (208.36,168.95) and (208.36,168.95) .. (208.36,168.95) .. controls (194.27,168.95) and (182.86,138.95) .. (182.86,101.95) .. controls (182.86,64.94) and (194.27,34.95) .. (208.36,34.95) ;  
\draw  [draw opacity=0] (208.36,34.95) .. controls (208.36,34.95) and (208.36,34.95) .. (208.36,34.95) .. controls (222.44,34.95) and (233.86,64.94) .. (233.86,101.95) .. controls (233.86,138.95) and (222.44,168.95) .. (208.36,168.95) -- (208.36,101.95) -- cycle ; \draw   (208.36,34.95) .. controls (208.36,34.95) and (208.36,34.95) .. (208.36,34.95) .. controls (222.44,34.95) and (233.86,64.94) .. (233.86,101.95) .. controls (233.86,138.95) and (222.44,168.95) .. (208.36,168.95) ;  
\draw  [draw opacity=0][dash pattern={on 4.5pt off 4.5pt}] (284.5,157) .. controls (284.5,157) and (284.5,157) .. (284.5,157) .. controls (274.84,157) and (267,132.15) .. (267,101.5) .. controls (267,70.85) and (274.84,46) .. (284.5,46) -- (284.5,101.5) -- cycle ; \draw  [dash pattern={on 4.5pt off 4.5pt}] (284.5,157) .. controls (284.5,157) and (284.5,157) .. (284.5,157) .. controls (274.84,157) and (267,132.15) .. (267,101.5) .. controls (267,70.85) and (274.84,46) .. (284.5,46) ;  
\draw  [draw opacity=0] (284.5,46) .. controls (284.5,46) and (284.5,46) .. (284.5,46) .. controls (294.16,46) and (302,70.85) .. (302,101.5) .. controls (302,132.15) and (294.16,157) .. (284.5,157) -- (284.5,101.5) -- cycle ; \draw   (284.5,46) .. controls (284.5,46) and (284.5,46) .. (284.5,46) .. controls (294.16,46) and (302,70.85) .. (302,101.5) .. controls (302,132.15) and (294.16,157) .. (284.5,157) ;  

\draw  [dash pattern={on 4.5pt off 4.5pt}]  (345,102) -- (22,102) ;
\draw [line width=0.75]    (345.95,102.02) -- (292.59,153.61) ;
\draw [line width=0.75]    (345.95,102.02) -- (290.32,48.67) ;
\draw [color={rgb, 255:red, 208; green, 2; blue, 27 }  ,draw opacity=1 ][line width=1.5]    (22,102) -- (231,75) ;
\draw [color={rgb, 255:red, 208; green, 2; blue, 27 }  ,draw opacity=1 ][line width=1.5]    (231,75) .. controls (259,72) and (273,72) .. (301,82) ;
\draw [color={rgb, 255:red, 208; green, 2; blue, 27 }  ,draw opacity=1 ][line width=1.5]    (345,102) -- (301,82) ;

\draw  [line width=3.75] [line join = round][line cap = round] (206,102) .. controls (206,102) and (206,102) .. (206,102) ;
\draw  [line width=3.75] [line join = round][line cap = round] (285,102) .. controls (285,102) and (285,102) .. (285,102) ;
\draw  [color={rgb, 255:red, 208; green, 2; blue, 27 }  ,draw opacity=1 ][line width=3.75] [line join = round][line cap = round] (231,75) .. controls (231,75) and (231,75) .. (231,75) ;
\draw  [color={rgb, 255:red, 208; green, 2; blue, 27 }  ,draw opacity=1 ][line width=3.75] [line join = round][line cap = round] (301,82) .. controls (301,82) and (301,82) .. (301,82) ;
\draw  [line width=3.75] [line join = round][line cap = round] (345,102) .. controls (345,102) and (345,102) .. (345,102) ;
\draw  [line width=3.75] [line join = round][line cap = round] (24,102) .. controls (24,102) and (24,102) .. (24,102) ;
\draw  [line width=3.75] [line join = round][line cap = round] (239.2,102) .. controls (239.2,102) and (239.2,102) .. (239.2,102) ;

\draw (4,93.4) node [anchor=north west][inner sep=0.75pt]    {$p$};
\draw (357,93.4) node [anchor=north west][inner sep=0.75pt]    {$q$};
\draw (106,41.4) node [anchor=north west][inner sep=0.75pt]    {$\mathcal{C}_{p}$};
\draw (324,52.4) node [anchor=north west][inner sep=0.75pt]    {$\mathcal{C}_{q}$};
\draw (239,5.4) node [anchor=north west][inner sep=0.75pt]    {$\partial \mathcal{O}$};
\draw (244,54.4) node [anchor=north west][inner sep=0.75pt]    {$\textcolor[rgb]{0.82,0.01,0.11}{u}$};
\draw (209,129.4) node [anchor=north west][inner sep=0.75pt]    {$\Gamma _{p}$};
\draw (252.59,130.01) node [anchor=north west][inner sep=0.75pt]    {$\Gamma _{q}$};

\end{tikzpicture}

\caption{Depicts the scenario where $\pa\cO=\bS^n$ and $l_{pq}$ intersects the origin. In this case, there are infinitely many minimizers.}

\end{figure}
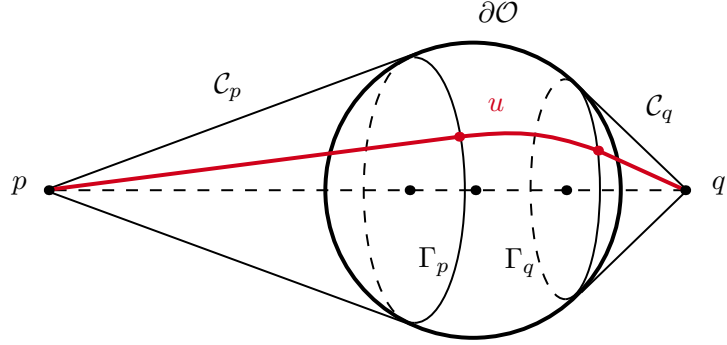

\begin{figure}\label{fig:unique}

\tikzset{every picture/.style={line width=0.75pt}} 

\begin{tikzpicture}[x=0.75pt,y=0.75pt,yscale=-1,xscale=1]

\draw  [line width=1.5]  (163.5,102) .. controls (163.5,60.85) and (196.85,27.5) .. (238,27.5) .. controls (279.15,27.5) and (312.5,60.85) .. (312.5,102) .. controls (312.5,143.15) and (279.15,176.5) .. (238,176.5) .. controls (196.85,176.5) and (163.5,143.15) .. (163.5,102) -- cycle ;
\draw  [line width=0.75]  (22,102) -- (209.86,32.62) -- (209.86,32.62) -- cycle ;
\draw [line width=0.75]    (22,102) -- (209.41,170.89) ;
\draw  [draw opacity=0][dash pattern={on 4.5pt off 4.5pt}] (208.36,168.95) .. controls (208.36,168.95) and (208.36,168.95) .. (208.36,168.95) .. controls (208.36,168.95) and (208.36,168.95) .. (208.36,168.95) .. controls (194.27,168.95) and (182.86,138.95) .. (182.86,101.95) .. controls (182.86,64.94) and (194.27,34.95) .. (208.36,34.95) -- (208.36,101.95) -- cycle ; \draw  [dash pattern={on 4.5pt off 4.5pt}] (208.36,168.95) .. controls (208.36,168.95) and (208.36,168.95) .. (208.36,168.95) .. controls (208.36,168.95) and (208.36,168.95) .. (208.36,168.95) .. controls (194.27,168.95) and (182.86,138.95) .. (182.86,101.95) .. controls (182.86,64.94) and (194.27,34.95) .. (208.36,34.95) ;  
\draw  [draw opacity=0] (208.36,34.95) .. controls (208.36,34.95) and (208.36,34.95) .. (208.36,34.95) .. controls (222.44,34.95) and (233.86,64.94) .. (233.86,101.95) .. controls (233.86,138.95) and (222.44,168.95) .. (208.36,168.95) -- (208.36,101.95) -- cycle ; \draw   (208.36,34.95) .. controls (208.36,34.95) and (208.36,34.95) .. (208.36,34.95) .. controls (222.44,34.95) and (233.86,64.94) .. (233.86,101.95) .. controls (233.86,138.95) and (222.44,168.95) .. (208.36,168.95) ;  
\draw  [draw opacity=0][dash pattern={on 4.5pt off 4.5pt}] (305.95,131.56) .. controls (305.95,131.56) and (305.95,131.56) .. (305.95,131.56) .. controls (300.89,133.91) and (286.33,113.27) .. (273.43,85.47) .. controls (260.52,57.66) and (254.17,33.22) .. (259.23,30.87) -- (282.59,81.21) -- cycle ; \draw  [dash pattern={on 4.5pt off 4.5pt}] (305.95,131.56) .. controls (305.95,131.56) and (305.95,131.56) .. (305.95,131.56) .. controls (300.89,133.91) and (286.33,113.27) .. (273.43,85.47) .. controls (260.52,57.66) and (254.17,33.22) .. (259.23,30.87) ;  
\draw  [draw opacity=0] (259.23,30.87) .. controls (259.23,30.87) and (259.23,30.87) .. (259.23,30.87) .. controls (259.23,30.87) and (259.23,30.87) .. (259.23,30.87) .. controls (264.29,28.52) and (278.85,49.16) .. (291.75,76.96) .. controls (304.65,104.77) and (311.01,129.21) .. (305.95,131.56) -- (282.59,81.21) -- cycle ; \draw   (259.23,30.87) .. controls (259.23,30.87) and (259.23,30.87) .. (259.23,30.87) .. controls (259.23,30.87) and (259.23,30.87) .. (259.23,30.87) .. controls (264.29,28.52) and (278.85,49.16) .. (291.75,76.96) .. controls (304.65,104.77) and (311.01,129.21) .. (305.95,131.56) ;  

\draw  [dash pattern={on 4.5pt off 4.5pt}]  (238,102) -- (22,102) ;
\draw [line width=0.75]    (339.96,57.89) -- (339.96,57.89) -- (306.99,130.41) ;
\draw [line width=0.75]    (340,58.18) -- (263.92,32.27) ;
\draw [color={rgb, 255:red, 208; green, 2; blue, 27 }  ,draw opacity=1 ][line width=1.5]    (22,102) -- (225.4,51.6) ;
\draw [color={rgb, 255:red, 208; green, 2; blue, 27 }  ,draw opacity=1 ][line width=1.5]    (225.4,51.6) .. controls (243,47.2) and (256.2,46) .. (276.2,48) ;
\draw [color={rgb, 255:red, 208; green, 2; blue, 27 }  ,draw opacity=1 ][line width=1.5]    (340,58.18) -- (276.2,48) ;
\draw  [line width=3.75] [line join = round][line cap = round] (206,102) .. controls (206,102) and (206,102) .. (206,102) ;
\draw  [line width=3.75] [line join = round][line cap = round] (282.6,83) .. controls (282.6,83) and (282.6,83) .. (282.6,83) ;
\draw  [color={rgb, 255:red, 208; green, 2; blue, 27 }  ,draw opacity=1 ][line width=3.75] [line join = round][line cap = round] (224.4,51.4) .. controls (224.4,51.4) and (224.4,51.4) .. (224.4,51.4) ;
\draw  [color={rgb, 255:red, 208; green, 2; blue, 27 }  ,draw opacity=1 ][line width=3.75] [line join = round][line cap = round] (275.8,48) .. controls (275.8,48) and (275.8,48) .. (275.8,48) ;
\draw  [line width=3.75] [line join = round][line cap = round] (339.2,58.6) .. controls (339.2,58.6) and (339.2,58.6) .. (339.2,58.6) ;
\draw  [line width=3.75] [line join = round][line cap = round] (24,102) .. controls (24,102) and (24,102) .. (24,102) ;
\draw  [line width=3.75] [line join = round][line cap = round] (238.2,102) .. controls (238.2,102) and (238.2,102) .. (238.2,102) ;
\draw  [dash pattern={on 4.5pt off 4.5pt}]  (340,58.18) -- (238,102) ;

\draw (4,93.4) node [anchor=north west][inner sep=0.75pt]    {$p$};
\draw (351.8,45.8) node [anchor=north west][inner sep=0.75pt]    {$q$};
\draw (106,41.4) node [anchor=north west][inner sep=0.75pt]    {$\mathcal{C}_{p}$};
\draw (302,26.4) node [anchor=north west][inner sep=0.75pt]    {$\mathcal{C}_{q}$};
\draw (239,5.4) node [anchor=north west][inner sep=0.75pt]    {$\partial \mathcal{O}$};
\draw (236.8,49.2) node [anchor=north west][inner sep=0.75pt]    {$\textcolor[rgb]{0.82,0.01,0.11}{u}$};
\draw (209,129.4) node [anchor=north west][inner sep=0.75pt]    {$\Gamma _{p}$};
\draw (270.19,115.21) node [anchor=north west][inner sep=0.75pt]    {$\Gamma _{q}$};

\end{tikzpicture}

\caption{Illustrates the scenario where $\pa\cO=\bS^n$ but $l_{pq}$ does not pass through the origin. In this case, the minimizer is unique.}

\end{figure}
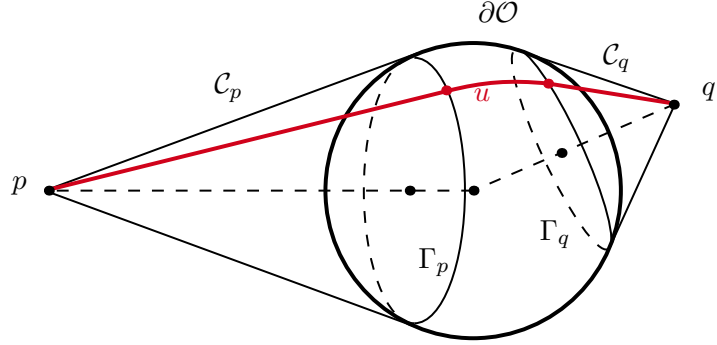

Let $p$ be a point such that $|p|>1$. If $\pa\cO=\bS^{n-1}$, then the vision cone $\mathcal{C}_p$ forms a regular cone, and the vision boundary $\Ga_p$ is a circle, which is the intersection of $\bS^{n-1}$ and a plane perpendicular to vector $p$. Note that any line segment connecting $p$ and a point in $\Ga_p$ has the same length. This also holds for $q$ with $|q|>1$. 

By \Cref{thm:charac}, the minimizer must be achieved by shortest geodesics connecting two vision boundaries $\Ga_p$ and $\Ga_q$. It should be noted that these two boundaries are parallel if and only if the line segment $l_{pq}$ passes through the origin. Furthermore, any geodesic on the sphere, forming a part of a great circle, is determined by a plane passing through the origin.

If $\Ga_p$ and $\Ga_q$ are parallel, then any geodesic that intersects the vision boundary orthogonally can be the geodesic part of minimizers. Since there are infinitely many such geodesics generated by rotation, we have infinite number of minimizers.
On the other hand, if $\Ga_p$ is not parallel to $\Ga_q$, then there exists a unique minimizing geodesic. This can be realized as one of the intersections between the plane passing through points $p$, $q$, and the origin, and the sphere $\bS^{n-1}$. 

Thus, the configurations of $p$ and $q$ that yield multiple minimizers occur when $p$, the origin, and $q$ are in order on a straight line. When the point $p$ is fixed, the non-uniqueness set $\mathcal{S}_2$ in \Cref{thm:unique} becomes one of the two components of $L\sm\overline\cO$, specifically, the half-line that does not include $p$.

\bigskip

We conclude this paper by suggesting the following conjecture:
\begin{conjecture}\label{conj}
Fix a point $p$ in $\R^n\sm \overline\cO$. 
Let $\mathcal{S}_{2}$ denotes the set of points $q$ in $\R^n\sm \overline \cO$ for which the minimizer of the energy \eqref{eq:E} over $\cA$ is not unique. Then the Hausdorff dimension of $\mathcal{S}_2$ is at most $n-1$.
\end{conjecture}

Note that if $\pa\cO$ is an ellipsoid in $\R^3$, then its cut locus can be one-dimensional, capable of generating a two-dimensional $\mathcal{S}_2$.


\section*{Acknowledgement}


Ki-Ahm Lee has been supported by National Research Foundation of Korea grant NRF-2020R1A2C1A01006256.
Taehun Lee has been supported by National Research Foundation of Korea grant RS-2023-00211258. The second author would like to thank Prof. Richard Schoen for introducing us to the paper \cite{Morgan78_Invent}.


\end{document}